\documentclass[reqno, 11pt]{amsart}
\usepackage{txfonts}%,pxfonts} %\lll fuer strikte ordnung \preceqq fuer partielle ordnung

\usepackage{bm}
\usepackage{amsmath}
\usepackage{amsfonts}
\usepackage{amssymb}
\usepackage{amsthm}
\usepackage{graphicx}
\usepackage{color}
\usepackage{enumitem}

\theoremstyle{plain}

\newtheorem{theorem}{Theorem}[section]

\newtheorem{example}[theorem]{Example}
\newtheorem{lemma}[theorem]{Lemma}

\newtheorem{remark}[theorem]{Remark}

\numberwithin{equation}{section}

\DeclareMathAlphabet\scr{U}{scr}{m}{n}
\SetMathAlphabet\scr{bold}{U}{scr}{b}{n}
  \DeclareFontFamily{U}{scr}{\skewchar\font'177}%
  \DeclareFontShape{U}{scr}{m}{n}{<-6>rsfs5<6-8>rsfs7<8->rsfs10}{}%
  \DeclareFontShape{U}{scr}{b}{n}{<-6>rsfs5<6-8>rsfs7<8->rsfs10}{}%

\renewcommand{\epsilon}{\varepsilon}
\renewcommand{\theta}{\vartheta}
\renewcommand{\rho}{\varrho}

\begin{document}
\title[]{How singular are moment generating functions?}
\author[E.~Mayerhofer]{Eberhard Mayerhofer}
\address{Vienna Institute of Finance, University of Vienna and Vienna University of Economics
and Business Administration, Heiligenst\"adterstrasse 46-48, 1190
Vienna, Austria} \email{eberhard.mayerhofer@vif.ac.at}
%\thanks{Thanks go to Gerard Letac and Alexander Smirnov (for a tipp that lead to the final
%form of section 3}.
\begin{abstract}
This short note concerns the possible singular behaviour of moment
generating functions of finite measures at the boundary of their
domain of existence. We look closer at Example 7.3 in O.
Barndorff-Nielsen's book {\it Information and Exponential Families
in Statistical Theory (1978)} and elaborate on the type of exhibited
singularity. Finally, another regularity problem is discussed and it
is solved through tensorizing two Barndorff- Nielsen's
distributions.
\end{abstract}
\subjclass[2000]{}  \maketitle
\section{Introduction}
Let $\mu$ be a finite positive measure on $\mathbb R^d$ $(d\geq 1)$,
and denote by
\[
G:\quad \mathbb R^d\rightarrow\mathbb R_+,\quad G(u)=\int_{\mathbb
R^d}e^{\langle u, \xi\rangle}\mu(d\xi)
\]
its moment generating function. (Here $\langle \cdot,\cdot\rangle$
denotes the standard Euclidean scalar product on $\mathbb R^d$.)
That is defined on the domain
\[
V:=\{u\in\mathbb R^d\mid G(u)<\infty\}.
\]
It is well known that $V$ is convex, and $u\mapsto G(u)$ is convex
thereon. Also, $G$ is analytic on the interior $V^\circ$ of $V$.

We  only talk about measures for which $V^\circ\neq\emptyset$.
Accordingly, we may assume without loss of generality that $0\in
V^\circ$ (use exponential tilting). Let us first have a look at the
behaviour of $G$ along half rays through the origin. Fix $u\in
\mathbb R^d$ and define
\[
\theta^*:=\sup\{\theta>0\mid G(\theta u)<\infty\}.
\]
Clearly we have
\[
\theta^*=\sup\{\theta>0\mid \theta u\in V\}=\sup\{\theta>0\mid
\theta u\in V^\circ\}\in(0,\infty].
\]
\begin{remark}\rm\label{remark1.1}
Either $\theta^*=\infty$ (not exciting), or $\theta^*<\infty$,
in which case the analytic function $G$ must exhibit a singularity
at $\theta^* u\in\partial V^\circ$. Along the ray $[0,\theta^* u]$,
two situations may occur:
\begin{enumerate}
\item \label{nonexotic} Either $\lim_{\theta\uparrow\theta^*}g(\theta u)=+\infty$, in
which case $\theta^* u\notin V$.
\item Or $\lim_{\theta\uparrow\theta^*}g(\theta u)=g(\theta^* u)<\infty$ where the equality follows from
Lebesgue's monotone convergence theorem. By definition,
$\theta^*u\in V$.
\end{enumerate}
\end{remark}
In the one-dimensional situation ($d=1$) it is clear that no other
limits then those along straight lines can be considered, hence all
possible singularities are of the above kind. However, the situation
is different in $\mathbb R^d$, $d\geq 2$. In the following section
we shall formulate a related non-trivial problem, and we recall a
partial answer from Barndorff-Nielsen's book \cite{barndorff} in
section \ref{sec partial answer} which involves a bivariate
distribution; that we then elaborate in greater detail. The final
section \ref{a new problem} poses a problem about singularities of
moment generating functions along the boundary of their existence
domains. That is solved by tensorizing Barndorff- Nielsen's
distribution with itself--hence it involves a four dimensional
distribution.\subsection{Problem}
%In a private communication with the author, Damir Filipovi\'c
%raised the following question.

Does there exist a probability measure $\mu$ on $\mathbb R^d$ whose
moment generating function $G$ has the following properties:
\begin{enumerate}
\item There exists a continuous curve $c: [0, 1]\to V, t\mapsto c(t)$ such
that
\item $G(c(1))\neq\lim_{t\uparrow 1} G(c(t))<\infty$, or, more
generally: There exists a sequence $t_k\uparrow 1$ such that the
sequence $(G(c(t_k)))_{k=1}^\infty$ has an accumulation point $p$ in
$\mathbb R_+ \cup\{\infty\}$ different from $G(c(1))$.
\end{enumerate}
Since moment generating functions are continuous along rays through
the origin, a solution $c$ of this problem necessarily needs to be
crooked; a solution is provided by Lemma \ref{lem answer} below.
\subsection{Motivational background}
In 2008, I had the pleasure to jointly elaborate with Damir
Filipovi\'c on a moment problem involving affine processes. It
concerned the question, for which real $u\in\mathbb R^d$ a
stochastic process $X$ on $D=\mathbb R_+^m\times\mathbb R^n$
satisfies the affine property
\[
\mathbb E^x[e^{\langle u, X_t\rangle}\mid
X_0=x]=e^{\phi(t,u)+\langle \psi(t,u),x\rangle}, \quad x\in D.
\]
We managed to completely characterize the validity of this affine
transform formula for affine diffusion processes\footnote{Note that
a-priori it is only clear that affine processes have a
characteristic function which is exponentially affine in the
state-variable. This is the key defining property of affine
processes. }: Either side is well defined, if the other is. In any
case, the expoenents $\phi,\psi$ solve a $(d+1)$ dimensional system
of Riccati differential equations (with initial values ($(0,u)$))
with blow up strictly beyond $t$\footnote{For Fourier-pricing
applications/implications, see \cite[Theorem 3.3]{ADPTA}.}.

A key finding which led to our characterization was that for any
$u\in\mathbb R^d$
\[
\mathbb E^x[e^{\langle \theta u, X_t\rangle}\mid X_0=x]\uparrow
\infty
\]
when
\[
\theta\uparrow\theta^*:=\sup\{\theta>0\mid \mathbb E^x[e^{\langle
\theta u, X_t\rangle}\mid X_0=x]<\infty\}.
\]
That is, the moment generating function of an affine diffusion does
not exhibit any exotic singularities, but only the one described in
Remark \ref{remark1.1} \ref{nonexotic}. If $D=\mathbb R_+$, then any
affine diffusion equals a square Bessel processes $X$, and therefore
for each $t$, $X_t$ is chi-square distributed (when appropriately
scaled), and in that case it is even obvious that the moment
generating function of the transition law of $X_t$ has a blow up
singularity at the boundary of its domain\footnote{Strictly
speaking, one needs to exclude deterministic motion, that is, one
should require that the diffusion coefficient of $X$ does not
vanish.}.

The problems of this paper have arisen naturally in the context of
our work\footnote{The question was posed to me by Damir Filipovi\'c
in early 2009.}. However, it turned out that a classification of
singularities of the moment generating function of processes which
exhibit jumps was not helpful for providing a characterization of
the affine property beyond the pure diffusion case. Also, such a
characterization seems to be not feasible (this is a subject believe
of the author). Ongoing work with Martin Keller-Ressel (Berlin)
provides a deeper understanding of the existence and non-existence
issues of exponential moments of affine jump-diffusions.
\section{A bivariate distribution}\label{sec partial answer}
In the following the notation for points in $\mathbb R^2$ with
coordinates $x,y$, namely $(x,y)$, should not be confused with open
intervals of the form $(a,b)=\{x\in\mathbb R\mid a<x<b\}$; the
reader will distinguish the notation easy from the context.

Barndorff-Nielsen used the following bivariate distribution $\mu$ on
$\mathbb R^2$ \cite[Example 7.3]{barndorff} defined by its density
$f(\xi):=\mu(d\xi)/d\xi$ (henceforth called Barndorff-Nielsen's
example or -function):
\begin{example}\label{barn ex}\rm
\[
f(\xi)=\frac{1}{2\sqrt{\pi}}(1+\xi_1^2)^{-3/2} e^{-\xi_1^2-\xi_2
^2/[4(1+\xi_1)^2]}.
\]
According to \cite[page 105]{barndorff}, for any
$u=(u_1,u_2))\in\mathbb R^2$ one has
\[
\int_{\mathbb R}e^{\langle
u,\xi\rangle}f(\xi)d\xi_2=(1+\xi_1^2)^{-1}e^{u_2^2+u_1\xi_1-(1-u_2^2)\xi_1^2}
\]
whence
\[
V=(\mathbb R\times (-1,1)) \cup \{(0,1), (0,-1)\}.
\]
That is, the domain of the moment generating functions consists of
an infinite strip of width $2$ parallel to the $\xi_1$-axis, and of
two isolated points on the $\xi_2$-axis.

For $u_2\in (-1,1)$ one has
\begin{align*}
G(u)&=I_1\times I_2\times I_3,\quad\textrm{where}\\ I_1(u)&=
e^{u_2^2+4u_1^2(1-u_2^2)},\\
I_2(u)&= e^{\frac{u_1}{4(1-u_2^2)}},\quad\textrm{and}\\
I_3(u)&=\int_{\mathbb R}(1+\xi_1^2)^{-1}\exp(-(1-u_2^2)\xi_1)d\xi_1.
\end{align*}
\end{example}
Barndorff-Nielsen shows that there exists a curve $c:
[0,t]\rightarrow V$ such that $c(0)=(0,0), c(1)=(0,1)$ but
\[
e\pi=G(c(1))\neq\lim_{t\uparrow 1} G(c(t))=\infty
\]
Here $e$ denotes the Euler number $e=2.71828\dots$. But much more
can be said about $G$. In a moment we reveal the following facts
\begin{lemma}\label{lem answer}
Let $\mathcal C$ be the set of continuous curves $c:
[0,t]\rightarrow V$ such that  $c(1)=(0,1)$. Then $G$ exhibits the
following (interrelated) properties:
\begin{enumerate}
\item \label{item1} There exists $c\in\mathcal C$ such that for each $p\in\mathbb [e\pi,\infty]$
there exists a sequence $(t_k)_{k=1}^\infty$ for which
\[
\lim_{k\rightarrow\infty} G(c(t_k))=p
\]
\item \label{item2} For each $p\in\mathbb [e\pi,\infty]$ there exists $c\in\mathcal C$
such that
\[
\lim_{t\uparrow 1} G(c(t))=p
\]
\item \label{item3} Let $(g^k)_{k=1}^\infty$ be any sequence in $[e\pi,\infty]$. Then there exists some $c\in\mathcal C$ and a sequence
$t_j\uparrow t$ such that $(g^k)_{k=1}^\infty$ equals the
accumulation points of the sequence $(G(c_{t_j}))_{j=1}^\infty$.
\end{enumerate}
\end{lemma}
\begin{proof}
Before we start the proof, let us note that for each $c\in\mathcal
C$ we have that
\begin{align}\label{eprop}
\lim_{t\uparrow 1} I_1(c(t))&=e,\\\label{piprop} \lim_{t\uparrow 1}
I_3(c(t))&=\int_{\mathbb R}(1+\xi_1^2)^{-1}d\xi_1=\pi.
\end{align}
So all we have to control is the behaviour of $I_2(c(t))$ as
$t\uparrow 1$.

Proof of \ref{item1}: The function
\[
h: [0, 1)\rightarrow (0,\infty),\quad
h(t):=\frac{\sin\left(\frac{1}{1-t}\right)+2-t}{1-t}
\]
has the following properties: (a) $h([0,1))=(0,\infty)$, (b) For
each $q\in[0,\infty]$ there exists a sequence $t_k\uparrow 1$ such
that $(h(t_k))_{k=1}^\infty$ has accumulation point $q$.

Defining
\[
c_2(t):=\sqrt{1-\frac{t}{4h(t)}}
\]
we introduce the curve $c(t):=(t,c_2(t))$. Then
\[
I_2(c(t))=\exp(h(t))
\]
and therefore by Property (b) (with $q=\log(p)$) as well as
\eqref{eprop}--\eqref{piprop} the claim \ref{item1} is proved.

Proof of \ref{item2}: $\log(p)>\log(e)=1$, hence we may choose
$c(t)=(t,c_2(t)=\sqrt{1-t/(4\log(p))})$; then $I_2(c(t))\equiv p$,
and therefore $G(c(t))\rightarrow I_1((0,1)) \times
I_2((0,1))=e\pi$.

Claim \ref{item3} follows directly from \ref{item1} by taking a
countable union $(t_j)_{j=1}^\infty$ of appropriate sequences
$(t^{j}_k)_k$, which for fixed $j$ converge to $g^j$.
\end{proof}

\section{Singularities along the boundary}\label{a new problem}
In Barndorff-Nielsen's example, the only boundary points contained
in  the domain of the moment generating function $V$ are (two)
isolated points, namely $(0,\pm 1)$. Hence we were not allowed to
look at regularity behaviour along the boundary. Here we pose a new
problem. Let's fix a ray (of flexible length) at the origin. With
its endpoint we let it strike along the boundary of $V$ to find jump
regularities of finite height:

Does there exist a probability measure $\mu$ on $\mathbb R^d$ whose
moment generating function $G$ has the following properties:
\begin{enumerate}
\item There exists a continuous curve $c: [0, 1]\to \partial V, t\mapsto c(t)$ such
that
\item $c([0,1))\subset V$.
\item $G(c(1))\neq\lim_{t\uparrow 1} G(c(t))<\infty$, or, more
generally: There exists a sequence $t_k\uparrow 1$ such that the
sequence $[G(c(t_k))]_{k=1}^\infty$ has an accumulation point $p$ in
$\mathbb R_+ \cup\{\infty\}$ different from $G(c(1))$.
\end{enumerate}
%Note that by requiring $p\neq G(c(1))$ we rule out the trivial
%constructions of finite measures $\mu(d\xi_1,d \xi_2)$ by mere
%tensorizing as follows. We define the following densities $\mu_1(d
%\xi_1)/d \xi_1:=\frac{1}{\pi(1+\xi_1^2)\exp(\vert\xi_1\vert}$,
%$\mu_2(d\xi_2)/d \xi_2:=\frac{1}{\exp(\vert\xi_2\vert}$. Then
%$\mu(d\xi):=\mu_1 (d\xi 1)\mu_2(d\xi_2) 1_{\mathbb R_+^2}(\xi)$ is a
%probability measure on $\mathbb R_+^2$ and
%\[
%V=(-\infty,1]\times (-\infty,1).
%\]
%Consider the curve
%\[
%c: [0,1]\rightarrow\mathbb R_+^2, \quad c(t)=(1,t)
%\]
%We have
%\[
%G(c(t))=G(1,t)= G(\mu_1)(1)\times G(\mu_2)(t)=1\times \frac{1}{1-t}
%\]
%
For a solution, we define a distribution $\mu$ on $\mathbb R^4$ by
its distribution function
 \[ f(\xi)=f(\xi_1,\xi_2) f(\xi_3,\xi_4)
 \]
 which is the product of two Barndorff-Nielsen functions.
Then the boundary of the domain of the moment generating function
for this distribution contains the set $V\times (0,1)$, where $V$ is
the domain of Barndorff-Nielsen m.g.f.  $G$. On this set, $G(u)$
equal to $e\pi G(u_1,u_2)$. Hence one may take any curve $\tilde
c(t)$ from Lemma \ref{lem answer} in $V$ and define the new curve \[
c: c(t)=(t,\tilde c_2(t), 0,1).
\]
Then all kinds of singularities are exhibited as $t\uparrow 1$.
Especially the above problem has a solution.

%$==============$ Let $O_\alpha$, $\alpha\in [0,\pi/2]$ be the family
%of clockwise rotations on $\mathbb R^2$ about the angle $\alpha$
%(measured in radiant). With respect to the canonical basis we may
%identify $O_\alpha$ with the coordinate matrix
%\[
%O_\alpha=\left(\begin{array}{ll}\cos(\alpha)&\sin(\alpha)\\-\sin(\alpha)&\cos(\alpha)\end{array}\right).
%\]
%Let $\alpha_k=\frac{\pi}{2}\frac{k-1}{k}$, $k=1,2,\dots$. Then
%$\alpha_1=0$ and $\alpha_k\uparrow \pi/2$ as $k\rightarrow\infty$.
%Let $\mu$ be the bivariate distribution from Example \ref{barn ex},
%and denote by $X$ a random variable on some (e.g. the canonical
%space) probability space $(\Omega,\mathcal F,\mathbb P)$,
%distributed according to $\mu$. We introduce the new measure $
%\nu(d\xi):=\mathbb P_Y $, where the r.v. $Y$ equals
%\[
%Y:=\lim_{N\rightarrow\infty}\frac{1}{N}\sum_{k=1}^\infty
%O_{-\alpha_k}(X_k),\quad X_k\sim X \quad i.i.d
%\]
%for any Lebesgue measurable set $A\subset\mathbb R^2$. The moment
%generating function of $\nu$ then equals
%\[
%G(\nu)(u)=\lim_{N\rightarrow\infty} \frac{1}{N}\sum_{k=1}^N
%G(O_{\alpha_k} u)
%\]
%Note that
%\[
%V(\nu)=\left(\cap_{k=1}^\infty O_{\alpha_k}(V) \right)\cup
%\{O_{\alpha_k}(0,1)\}
%\]
%Where does this go?
%
%Maybe some funny example can be constructed by "summing" over a
%continuum of orthogonal transformations of $X$, say
%$O(\alpha(X_\alpha))$, where $X_\alpha$ are all independent r.v.,
%and $\alpha\in (0,\pi/2)$?
%

%\bibliography{plain}

\end{document}